\documentclass[12pt,reqno]{amsart}
\usepackage{latexsym}
\usepackage{amssymb}
\usepackage{mathrsfs}
\usepackage{amsmath}
\usepackage{fancybox,color}
\usepackage{enumerate}
\usepackage[latin1]{inputenc}
\usepackage[colorlinks=true, linkcolor=magenta, citecolor=magenta]{hyperref}

\makeatletter
\@namedef{subjclassname@2020}{\textup{2020} Mathematics Subject Classification}
\makeatother

\usepackage{color}
{\catcode`p =12 \catcode`t =12 \gdef\eeaa#1pt{#1}}      
\def\accentadjtext#1{\setbox0\hbox{$#1$}\kern   
                \expandafter\eeaa\the\fontdimen1\textfont1 \ht0 }
\def\accentadjscript#1{\setbox0\hbox{$#1$}\kern 
                \expandafter\eeaa\the\fontdimen1\scriptfont1 \ht0 }
\def\accentadjscriptscript#1{\setbox0\hbox{$#1$}\kern   
                \expandafter\eeaa\the\fontdimen1\scriptscriptfont1 \ht0 }
\def\accentadjtextback#1{\setbox0\hbox{$#1$}\kern       
                -\expandafter\eeaa\the\fontdimen1\textfont1 \ht0 }
\def\accentadjscriptback#1{\setbox0\hbox{$#1$}\kern     
                -\expandafter\eeaa\the\fontdimen1\scriptfont1 \ht0 }
\def\accentadjscriptscriptback#1{\setbox0\hbox{$#1$}\kern 
                -\expandafter\eeaa\the\fontdimen1\scriptscriptfont1 \ht0 }
\def\itoverline#1{{\mathsurround0pt\mathchoice
        {\rlap{$\accentadjtext{\displaystyle #1}
                \accentadjtext{\vrule height1.593pt}
                \overline{\phantom{\displaystyle #1}
                \accentadjtextback{\displaystyle #1}}$}{#1}}
        {\rlap{$\accentadjtext{\textstyle #1}
                \accentadjtext{\vrule height1.593pt}
                \overline{\phantom{\textstyle #1}
                \accentadjtextback{\textstyle #1}}$}{#1}}
        {\rlap{$\accentadjscript{\scriptstyle #1}
                \accentadjscript{\vrule height1.593pt}
                \overline{\phantom{\scriptstyle #1}
                \accentadjscriptback{\scriptstyle #1}}$}{#1}}
        {\rlap{$\accentadjscriptscript{\scriptscriptstyle #1}
                \accentadjscriptscript{\vrule height1.593pt}
                \overline{\phantom{\scriptscriptstyle #1}
                \accentadjscriptscriptback{\scriptscriptstyle #1}}$}{#1}}}}

\newcommand{\ed}{\color{black}} 
\newcommand{\iol}{\itoverline}

\usepackage[left=1.8cm,top=2.3cm,right=1.8cm]{geometry}

\DeclareMathOperator{\diam}{diam}

\newtheorem{theorem}{Theorem}[section]
\newtheorem{corollary}[theorem]{Corollary}
\newtheorem{lemma}[theorem]{Lemma}

\theoremstyle{definition}
\newtheorem{definition}[theorem]{Definition}

\newcommand{\rad}{\operatorname{rad}}

\newcommand{\R}{{\mathbb R}}

\newcommand{\N}{{\mathbb N}}
\newcommand{\Z}{{\mathbb Z}}

\newcommand{\Ha}{{\mathcal H}}

\DeclareMathOperator{\ldima}{\underline{dim}_A}
\DeclareMathOperator{\ucodima}{\overline{co\,dim}_A}

\DeclareMathOperator{\Cp}{cap}

\newcommand{\ch}[1]{{\mbox{\raise 1pt\hbox{\Large$\chi$}}}_{\lower 1pt\hbox{$\scriptstyle #1\,$}}}

\def\1{\raisebox{2pt}{\rm{$\chi$}}}

\def\vint_#1{\mathchoice%
        {\mathop{\kern 0.2em\vrule width 0.6em height 0.69678ex depth -0.58065ex
                \kern -0.8em \intop}\nolimits_{\kern -0.4em#1}}%
        {\mathop{\kern 0.1em\vrule width 0.5em height 0.69678ex depth -0.60387ex
                \kern -0.6em \intop}\nolimits_{#1}}%
        {\mathop{\kern 0.1em\vrule width 0.5em height 0.69678ex
            depth -0.60387ex
                \kern -0.6em \intop}\nolimits_{#1}}%
        {\mathop{\kern 0.1em\vrule width 0.5em height 0.69678ex depth -0.60387ex
                \kern -0.6em \intop}\nolimits_{#1}}}
\def\vintslides_#1{\mathchoice%
        {\mathop{\kern 0.1em\vrule width 0.5em height 0.697ex depth -0.581ex
                \kern -0.6em \intop}\nolimits_{\kern -0.4em#1}}%
        {\mathop{\kern 0.1em\vrule width 0.3em height 0.697ex depth -0.604ex
                \kern -0.4em \intop}\nolimits_{#1}}%
        {\mathop{\kern 0.1em\vrule width 0.3em height 0.697ex depth -0.604ex
                \kern -0.4em \intop}\nolimits_{#1}}%
        {\mathop{\kern 0.1em\vrule width 0.3em height 0.697ex depth -0.604ex
                \kern -0.4em \intop}\nolimits_{#1}}}

\newcommand{\intav}{\vint}
\newcommand{\aveint}[2]{\mathchoice%
        {\mathop{\kern 0.2em\vrule width 0.6em height 0.69678ex depth -0.58065ex
                \kern -0.8em \intop}\nolimits_{\kern -0.45em#1}^{#2}}%
        {\mathop{\kern 0.1em\vrule width 0.5em height 0.69678ex depth -0.60387ex
                \kern -0.6em \intop}\nolimits_{#1}^{#2}}%
        {\mathop{\kern 0.1em\vrule width 0.5em height 0.69678ex depth -0.60387ex
                \kern -0.6em \intop}\nolimits_{#1}^{#2}}%
        {\mathop{\kern 0.1em\vrule width 0.5em height 0.69678ex depth -0.60387ex
                \kern -0.6em \intop}\nolimits_{#1}^{#2}}}

\newcommand{\dist}{\operatorname{dist}}

\usepackage{ifthen}
\newcommand{\formula}[2][nolabel]%
{%
 \ifthenelse{\equal{#1}{nolabel}}%
 {\begin{align*} #2 \end{align*}}%
 {%
  \ifthenelse{\equal{#1}{}}%
  {\begin{align} #2 \end{align}}%
  {\begin{align} \label{#1} \begin{aligned} #2 \end{aligned} \end{align}}%
 }%
}

\title[Fractional Poincar\'e and Hardy inequalities]{Fractional Poincar\'e and localized Hardy\\ inequalities on metric spaces}

\subjclass[2020]{
26D15   
(31C15,	
35A23,  
46E36)} 

\author[B.\! Dyda]{Bart{\l}omiej Dyda}   
\address[B.D.]{Faculty of Pure and Applied Mathematics\\ Wroc{\l}aw University 
	of Science and Technology\\
	Wybrze\.ze Wyspia\'nskiego 27,
	50-370 Wroc{\l}aw, Poland
}
\email{bdyda@pwr.edu.pl\qquad dyda@math.uni-bielefeld.de}
\thanks{B.D. was partially supported by grant NCN 2015/18/E/ST1/00239.}

\author[J.\! Lehrb\"ack]{Juha Lehrb\"ack}

\address[J.L.]{University of Jyvaskyla, Department of Mathematics and Statistics, P.O. Box 35, FI-40014 University of Jyvaskyla, Finland} \email{juha.lehrback@jyu.fi}

\author[A.V.\! V\"ah\"akangas]{Antti V. V\"ah\"akangas}

\address[A.V.V.]{University of Jyvaskyla, Department of Mathematics and Statistics, P.O. Box 35, FI-40014 University of Jyvaskyla, Finland} \email{antti.vahakangas@iki.fi}

\date{\today}

\begin{document}

\begin{abstract}
We prove fractional Sobolev--Poincar\'e inequalities, 
capacitary versions of fractional Poincar\'e inequalities,
and pointwise and localized fractional Hardy inequalities
in a metric space equipped with a doubling measure.
Our results generalize and extend earlier work
where such inequalities have been considered 
in the Euclidean spaces or in the non-fractional
setting in metric spaces. The results concerning
pointwise and localized variants of fractional Hardy inequalities
are new even in the Euclidean case.
\end{abstract}

\maketitle

\section{Introduction}

Let $X=(X,d,\mu)$ be a metric measure space
and let $1\le p,q,t<\infty$ and $0<s<1$. 
The fractional $(s,q,p,t)$-Poincar\'e (or Sobolev--Poincar\'e)
inequality on $X$ 
reads as
\begin{equation}\label{e.poinc_intro}
\biggl(\intav_B \lvert u(x)-u_B\rvert^q\,dx\biggr)^{1/q}
\le c_Pr^{s}\biggl(\intav_{\lambda B}\biggl( \int_{\lambda B}\frac{|u(x)-u(y)|^t}{d(x,y)^{st}\mu(B(x,d(x,y)))}\,dy\biggr)^{p/t}\,dx\biggr)^{1/p},
\end{equation}
where $dx=d\mu(x)$ and $dy=d\mu(y)$.
We say that $X$  
supports a $(s,q,p,t)$-Poincar\'e inequality if
there are constants $c_P>0$ and $\lambda \ge 1$ such that 
inequality~\eqref{e.poinc_intro}
holds for every ball $B=B(x_0,r)\subset X$
and for all functions $u\colon X\to \R$ that are
integrable on balls. 

If $q\le \min\{p,t\}$ and the measure $\mu$ is doubling, 
then it is straightforward to show that the space $X$ 
supports a $(s,q,p,t)$-Poincar\'e inequality;
see Lemma~\ref{l.qp}. 
This is quite different compared to the usual (i.e.\ non-fractional) Poincar\'e inequalities,
whose validity in a metric measure space is usually an indication of
the existence of a rich geometric structure in the space;
we refer to the monographs~\cite{MR2867756,HeinonenKoskelaShanmugalingamTyson2015}
for more explanation and examples.

The main goal in this work is to prove stronger variants of
fractional inequalities, 
such as (Sobolev--)Poincar\'e inequalities for $q>p$, capacitary versions of Poincar\'e inequalities,
and pointwise and localized Hardy inequalities.
The validity of these stronger variants
often requires additional assumptions on the space and the functions and 
sets in the
inequalities. For example, in the so-called boundary Poincar\'e inequalities
the mean value $u_B$ on the left-hand side of~\eqref{e.poinc_intro}
can be omitted if the set where $u=0$ (i.e.\ the ``boundary'')
is large enough. 

The parameter $1\le  t<\infty$
in inequality~\eqref{e.poinc_intro} 
allows certain flexibility in the applications,
for instance in the proof of
the localized fractional Hardy inequality
\begin{equation}\label{eq.local_hardy_intro}
\int_{B\setminus E} \frac{\lvert u(x)\rvert^p}{d(x,E)^{sp}}\,dx
\le C \int_{\lambda B}\int_{\lambda B}
\frac{\lvert u(x)-u(y)\rvert^p}{d(x,y)^{sp}\mu(B(x,d(x,y)))}\,dy\,dx,
\end{equation}
where $0<s<1$, $1<p<\infty$, $\lambda\ge 1$,
$E\subset X$ is a closed set,
$B=B(w,r)$ for some $w\in E$ and $0<r<\diam(E)$, and 
$u\colon X\to\R$ is a continuous function with $u=0$ on $E$.
As one of our main results we show that the validity of inequality~\eqref{eq.local_hardy_intro}
is essentially characterized by dimensional information related to the set $E$.
More precisely, $\ucodima(E) < sp$ is sufficient and $\ucodima(E)\le sp$ is necessary
for~\eqref{eq.local_hardy_intro}, where $\ucodima(E)$ is the upper Assouad codimension of
$E$, see Definition~\ref{def.ucodima}. 
The upper bound for this codimension means that the set $E$
must be sufficiently large in comparison to the size of the ambient space $X$.  
In the Euclidean case $X=\R^n$ we have $\ucodima(E)=n-\ldima(E)$, where $\ldima$ is the
lower dimension (or lower Assouad dimension) of $E\subset\R^n$.



We obtain the localized inequality~\eqref{eq.local_hardy_intro} as a consequence of
a pointwise fractional Hardy inequality, given in terms of a maximal operator.
In the non-fractional case in $\R^n$, 
pointwise Hardy inequalities were introduced in~\cite{Hajlasz1999} and~\cite{KinnunenMartio1996}.
Sufficient and necessary
conditions for pointwise Hardy inequalities in metric spaces have been given 
in~\cite{KorteEtAl2011};
see also~\cite{Lehrback2014} for weighted variants.
Fractional Hardy inequalities on open sets
have been studied 
in the Euclidean space $\R^n$
for instance in~\cite{Dyda2004,DydaKijaczko,DydaVahakangas2015,EdmundsHurriSyrjanenVahakangas2014,Ihnatsyeva_et_al2014,LossSloane2010}
and in general metric spaces in~\cite{DydaEtAl2019,DydaVahakangas2014},
but the present pointwise and localized versions of fractional Hardy inequalities,
as well as the boundary Poincar\'e inequalities, are new even in the Euclidean case.
Sobolev--Poincar\'e
inequalities have also been considered in
more general sets than balls, 
in particular in the so-called John domains,
see~\cite{MR4018755,HurriSyrjanenVahakangas2013,HurriSyrjanenVahakangas2015}
and the references therein.

The outline for the rest of the paper is as follows.
In Section~\ref{s.preli} we review the necessary
definitions and notation on metric measure spaces
and give in Lemma~\ref{l.qp} the basic versions of 
fractional Poincar\'e inequalities for $q\le p$;
these are used as a starting point in the proofs of the stronger inequalities
in the subsequent sections.
Section~\ref{s.sobo} is devoted to extending the range in the
fractional (Sobolev--)Poincar\'e inequalities to $q>p$, following the ideas in the proofs
of the corresponding fractional results in the Euclidean case~\cite{DIV} as well in the non-fractional
results in metric spaces~\cite{MR2867756}.
In Section~\ref{s.capacitary} we introduce a variant of the relative fractional
capacity and prove a Maz$'$ya type capacitary Poincar\'e inequality in Theorem~\ref{t.Mazya}. 
Boundary Poincar\'e inequalities are obtained in Theorem~\ref{t.assouad}
and Corollary~\ref{cor.bdry_poinc} under
the dimensional condition $\ucodima(E) < sp$,
which is connected to the relative capacity via suitable Hausdorff contents;
see Definition~\ref{def.Hcont} and Lemma~\ref{l.codim}.
In Section~\ref{s.hardy}, the localized Hardy inequality~\eqref{eq.local_hardy_intro} 
is obtained in Theorem~\ref{t.integrated}
as a consequence of a pointwise fractional Hardy inequality, see Theorem~\ref{t.pointwise},
which in turn is based on the boundary Poincar\'e inequality in Theorem~\ref{t.assouad}.
Theorem~\ref{t.converse} then shows the necessity of the condition $\ucodima(E) \le sp$
for the localized inequality~\eqref{eq.local_hardy_intro}. 
In Sections~\ref{s.capacitary} and~\ref{s.hardy}  
our proofs often follow the main lines of the 
proofs from the non-fractional case,
as for instance in~\cite{MR2867756,HeinonenKoskelaShanmugalingamTyson2015,KorteEtAl2011}, but
due to the non-locality of the setting several modifications are needed in the proofs.

\section{Preliminaries}\label{s.preli}

We assume throughout this paper that $X=(X,d,\mu)$ is 
a metric measure space (with at least two points),  
where
$\mu$ is a Borel measure supported on $X$ such that 
$\mu(\{x\})=0$ for all $x\in X$ and
$0<\mu(B)<\infty$ for
all (open) balls 
\[B=B(x,r):= \{y\in X : d(x,y)< r\}\]
with $x\in X$ and $r>0$. We make the tacit assumption that each ball $B\subset X$ has a fixed center $x_B$ 
and radius $\rad(B)$, and thus notation such as $\lambda B = B(x_B,\lambda \rad(B))$ 
is well-defined for all $\lambda>0$. When $E,F\subset X$, we let $\diam(E)$ denote the diameter of $E$ and $\dist(E,F)$ is the distance between the sets $E,F\subset X$. We  use $d(x,E)=\dist(x,E)=\dist(\{x\},E)$ to denote the distance from a point $x\in X$ to the set $E$.
If $E\subset X$, then $\ch{E}$
denotes the characteristic function of $E$; that is, $\ch{E}(x)=1$ if $x\in E$
and $\ch{E}(x)=0$ if $x\in X\setminus E$.

We also assume throughout
that $\mu$ is \emph{doubling}, that is, there is a constant $c_D\ge  1$ such that 
whenever $x\in X$ and $r>0$, we have
\begin{equation}\label{e.doubling}
  \mu(B(x,2r))\le c_D\, \mu(B(x,r)).
\end{equation}
Iteration of~\eqref{e.doubling} shows that if $\mu$ is doubling, then there exist an exponent $Q>0$ and a constant $c_Q>0$,
both only depending on $c_D$, such that the quantitative doubling condition
\begin{equation}\label{e.doubling_quant}
  \frac{\mu(B(y,r))}{\mu(B(x,R))}\ge c_Q\Bigl(\frac {r}{R}\Bigr)^Q
\end{equation}
holds whenever $y\in B(x,R)\subset X$ and $0<r<R$. 
Condition~\eqref{e.doubling_quant} always holds
for $Q\ge \log_2c_D$,  but it can hold for smaller values of $Q$ as well. See~\cite[Lemma~3.3]{MR2867756}
for details. 

In some of our results we also need
to assume that $\mu$ is
\emph{reverse doubling}, in the sense that 
there are constants  $0<\kappa<1$ \ed and $0<c_R<1$ such that
\begin{equation}\label{e.rev_dbl_decay}
\mu(B(x,\kappa r))\le c_R\, \mu(B(x,r))
\end{equation}
for every $x\in X$ and 
$0<r<\diam(X)/2$.
If $X$ is connected and $0<\kappa<1$, then
inequality~\eqref{e.rev_dbl_decay} follows from the doubling property~\eqref{e.doubling} 
with $0<c_R=c_R(c_D,\kappa)<1$. See for instance~\cite[Lemma~3.7]{MR2867756}.
Iteration of~\eqref{e.rev_dbl_decay} shows that
if $\mu$ is reverse doubling, then there exist
an exponent $\sigma>0$ and a constant $c_\sigma>0$,
both only depending on $\kappa$ and $c_R$, such that the quantitative 
 reverse doubling condition
\begin{equation}\label{e.reverse_doubling}
 \frac{\mu(B(x,r))}{\mu(B(x,R))} \le c_\sigma\Bigl(\frac{r}{R}\Bigr)^\sigma
\end{equation}
holds for every $x\in X$ and 
$0<r<R<2\diam(X)$.

If the measure  $\mu$ is Ahlfors $Q$-regular for some $Q>0$, that is, there is a constant $C>0$
such that 
\[
\frac 1 C r^Q\le \mu(B(x,r))\le Cr^Q
\] 
for every $x\in X$ and $0<r<\diam(X)$,
then $\mu$ is both doubling and reverse doubling, and 
 the quantitative estimates \eqref{e.doubling_quant} and~\eqref{e.reverse_doubling} hold with
the  exponent $Q$.


We abbreviate $d\mu(x)=dx$ and say that a function $u\colon X\to\R$ is integrable
on balls, if $u$ is $\mu$-measurable and
\[\lVert u\rVert_{L^1(B)}=\int_B \lvert u(x)\rvert\,dx<\infty\] for all
balls $B\subset X$. In particular, for such functions the
integral average
\[
u_B = \intav_B u(x)\,dx =\frac{1}{\mu(B)}\int_B u(x)\,dx
\]
is well-defined whenever $B$ is a ball in $X$.
Observe that we do not  always assume that the space $X$ is complete, 
and hence continuous functions are not necessarily integrable on balls.

\begin{definition}\label{def.poinc}
Let $1\le p,q,t<\infty$ and $0<s<1$. 
We say that $X$  
supports a $(s,q,p,t)$-Poincar\'e inequality, if
there are constants $c_P>0$ and $\lambda \ge 1$ such that inequality
\begin{equation}\label{e.poinc}
\biggl(\intav_B \lvert u(x)-u_B\rvert^q\,dx\biggr)^{1/q}
\le c_Pr^{s}\biggl(\intav_{\lambda B}\biggl( \int_{\lambda B}\frac{|u(x)-u(y)|^t}{d(x,y)^{st}\mu(B(x,d(x,y)))}\,dy\biggr)^{p/t}\,dx\biggr)^{1/p}
\end{equation}
holds for every ball $B=B(x_0,r)\subset X$
and for all functions $u\colon X\to \R$  that are
integrable on balls. 
\end{definition}

In particular the left-hand side of \eqref{e.poinc}
is finite, if the right-hand side is finite.

If $u\colon X\to\R$ is a measurable function, $0<s<1$, $1\le t<\infty$, and $A\subset X$ is a measurable set, we write 
\[
g_{u,s,t,A}(x)=\biggl(\int_{A} \frac{\vert u(x)-u(y)\vert^t}{d(x,y)^{st}\mu(B(x,d(x,y)))}\,dy\biggr)^{1/t},\qquad \text{ for every } x\in X.
\]
Using this notation, 
the $(s,q,p,t)$-Poincar\'e inequality \eqref{e.poinc}
can be written as
\begin{equation*}
\biggl(\intav_B \lvert u(x)-u_B\rvert^q\,dx\biggr)^{1/q}
\le c_Pr^{s}\biggl(\intav_{\lambda B}g_{u,s,t,\lambda B}(x)^p\,dx\biggr)^{1/p}.
\end{equation*}
We will repeatedly use the 
facts that $g_{\lvert u\rvert,s,t,A}\le g_{u,s,t,A}$ and $g_{u,s,t,A}\le g_{u,s,t,A'}$ when $A\subset A'$.

The following lemma shows that $X$ supports a $(s,q,p,t)$-Poincar\'e inequality
if $1\le q \le \min\{p,t\}$.  
We emphasize that \ed the doubling condition
on $\mu$ is the only quantitative property of $X$ that
is needed in this case. 
This result is certainly 
known among experts,  but  we include
the short proof for the \ed convenience of the reader. We refer to \cite[Lemma 2.2]{HurriSyrjanenVahakangas2013} for a
variant of this result in $\R^n$.

\begin{lemma}\label{l.qp}
Assume that $1\le q,p,t<\infty$, $q\le \min\{p,t\}$, and $0<s<1$. Then
$X$ supports the $(s,q,p,t)$-Poincar\'e inequality \eqref{e.poinc}
with constants $\lambda=1$ and $c_P=c_P(s,q,t,c_D)$.
\end{lemma}

\begin{proof}
Fix a ball $B=B(x_0,r)\subset X$ and a 
function $u\colon X\to \R$ that is integrable on balls.
Then
\begin{align*}
\vint_{B} \lvert u(x)-u_B\rvert^q\,dx
&\le \vint_{B}\intav_B \lvert u(x)-u(y)\rvert^q\,dy\,dx\\
&\le  \vint_{B}\biggl(\intav_B \lvert u(x)-u(y)\rvert^t\,dy\,\biggr)^{q/t}dx\\
&\le   \biggl(\vint_{B}\biggl(\intav_B \lvert u(x)-u(y)\rvert^t\,dy\,\biggr)^{p/t}dx\biggr)^{q/p}\\
&\le   r^{sq}\biggl(\vint_{B}\biggl(\int_B \frac{\lvert u(x)-u(y)\rvert^t}{r^{st}\mu(B)}\,dy\,\biggr)^{p/t}dx\biggr)^{q/p}\\
&\le   Cr^{sq}\biggl(\vint_{B}\biggl(\int_B \frac{\lvert u(x)-u(y)\rvert^t}{d(x,y)^{st}\mu(4B)}\,dy\,\biggr)^{p/t}dx\biggr)^{q/p}\\
&\le   Cr^{sq}\biggl(\vint_{B}\biggl(\int_B \frac{\lvert u(x)-u(y)\rvert^t}{d(x,y)^{st}\, \mu(B(x,d(x,y)))}\,dy\,\biggr)^{p/t}dx\biggr)^{q/p}.
\end{align*}
This yields the desired inequality
\eqref{e.poinc} with $\lambda=1$ and
$c_P=c_P(s,q,t,c_D)$.
\end{proof}

\section{Sobolev--Poincar\'e inequalities}\label{s.sobo}

As with the usual Poincar\'e inequalities (see~\cite{MR2867756,HeinonenKoskelaShanmugalingamTyson2015}), 
also in the fractional case it is
possible to improve inequalities from the case $q\le p$ (in Lemma~\ref{l.qp}) to the case $q>p$,
up to the ``Sobolev exponent'' $p^*=Qp/(Q-sp)$; see 
Theorem~\ref{t.qp_impro} below. 
For this purpose, we apply
a metric measure space version of the fractional truncation method in
\cite[Proposition~5]{DydaVahakangas2015}, \cite[Theorem~4.1]{DIV};  see also
\cite[Proposition~2.14]{MR4046971}. 
In the proof we need the following 
auxiliary result, which is a special case of \cite[Lemma~5]{MR1886617}.

\begin{lemma}\label{MeasureLem}
Assume  that $g\ge 0$ is a measurable function on a ball $B\subset X$ with
\[\mu(\{x\in B : g(x)=0\})\geq \mu(B)/2.\]
Then inequality
\[\mu (\{x\in B : g(x)>t\})\le 2\inf_{a\in\R}\mu(\{x\in B : \lvert g(x)-a\rvert>t/2\})
\]
holds  for every $t>0$.
\end{lemma}

Theorem~\ref{t.truncation} below is metric measure space version of
the Euclidean result in~\cite[Theorem~4.1]{DIV}. 
We will later apply this theorem
with the kernel 
\[K(y,z)=\frac{1}{d(y,z)^{sp}\mu(B(y,d(y,z)))},\qquad y,z\in X,\] 
but we formulate the result in terms of general kernels.
The proof is a straightforward adaptation of the proof in~\cite{DIV},
and it is based on a fractional Maz$'$ya truncation method. 

\begin{theorem}\label{t.truncation}
Let $0<s <1$, $0<p\le q<\infty$, 
and $\lambda\ge 1$. 
Let $K\colon X\times X\to [0,\infty]$ be a measurable function
and let $B=B(x_0,r)\subset X$ be a ball. 
Then the following conditions are equivalent: 
\begin{itemize}
\item[(A)]
There is a constant $C_1>0$ such that inequality
\begin{align*}
\inf_{a\in\R}\sup_{t>0} \mu ( & \{x\in B: |u(x)-a|>t\} ) t^{q}
\le C_1
\biggl(\int_{\lambda B}\int_{\lambda B} \vert u(y)-u(z)\vert^p K(y,z)
\,dz\,dy\biggr)^\frac{q}{p}
\end{align*}
holds for  every $u\in L^\infty(\lambda B)$.
\item[(B)]
There is a constant $C_2>0$ such that inequality
\begin{align*}
\inf_{a\in\R} \int_B\vert u(x)-a\vert ^{q}\,dx
\le C_2
\biggl(\int_{\lambda B}\int_{\lambda B}\vert u(y)-u(z)\vert^p K(y,z)
\,dz\,dy\biggr)^\frac{q}{p}
\end{align*}
holds for every $u\in L^1(\lambda B)$, and
the left-hand side is finite
if the right-hand side is finite.
\end{itemize}
Moreover, in 
the implication from {\rm (A)} to {\rm (B)} the constant $C_2$ is of the form $C(p,q)C_1$,
 and in the implication from {\rm (B)} to {\rm (A)} we have
$C_1=C_2$.
\end{theorem}

\begin{proof}
The implication from (B) to (A) with $C_1=C_2$ follows from Chebyshev's inequality.
Let us then assume that condition (A) holds.
Fix $u\in L^1(\lambda B)$ and let $b\in\R$ be such that
\begin{equation}\label{e.apu}
\mu (\{x\in B : u(x)\geq b\}) \geq\frac{\mu(B)}{2}\quad\text{and}\quad\mu (\{x\in B : u(x)\leq b\})\geq\frac{\mu(B)}{2}.
\end{equation}
We write $v_{+}=\max\{u-b,0\}$ and  $v_{-}=-\min\{u-b,0\}$. In the sequel $v$ denotes either $v_{+}$ or $v_{-}$;
all the statements are valid in both cases. Moreover, without loss of generality,
we may assume that $v\ge 0$ is defined and finite everywhere in $\lambda B$.

For  $0<t_1<t_2<\infty$ and every $x\in \lambda B$, we define
\[
v_{t_1}^{t_2}(x)=
\begin{cases}
t_2-t_1, \qquad &\text{if } t_2\le v(x),\\
v(x)-t_1, &\text{if } t_1<v(x)<t_2,\\
0, &\text{if } v(x)\leq t_1.
\end{cases}
\]
Observe from \eqref{e.apu} that
\[
\mu (\{x\in B : v_{t_1}^{t_2}(x)=0\})\geq \mu(B)/2.\]
By Lemma \ref{MeasureLem} and condition  (A), both applied to the non-negative function $v_{t_1}^{t_2}\in L^\infty(\lambda B)$,
\begin{equation}\label{e.mas}
\begin{split}
\sup_{t>0} \mu (\{x\in B : v_{t_1}^{t_2}(x)>t\})
\,t^{q}&\le 2^{1+q}\inf_{a\in\R}\sup_{t>0}\mu (\{x\in B : \lvert v_{t_1}^{t_2}(x)-a\rvert >t\})\,t^q \\
&\le 2^{1+q}C_1 \biggl(\int_{\lambda B}\int_{\lambda B}\lvert v_{t_1}^{t_2}(y)-v_{t_1}^{t_2}(z)\rvert^p K(y,z)\,dz\,dy
\biggr)^\frac{q}{p}.
\end{split}
\end{equation}

We write
$E_k = \{x\in \lambda B  : v(x) > 2^{k} \}$ and $A_k = E_{k-1} \setminus E_{k}$, where $k\in \Z$.
Since $v\ge 0$ is finite everywhere in $B$,  we can write
\begin{equation}\label{e.decomp}
B= \{x\in B :  0\le v(x)<\infty\} = \Biggl(\bigcup_{i\in \Z}
B\cap A_i\Biggr) \cup \Bigl(B\cap \underbrace{\{x\in \lambda B  :  v(x)=0\}}_{=:A_{-\infty}})\Bigr).
\end{equation}
Hence, by inequality \eqref{e.mas} and the fact that $\sum_{k\in \Z} |a_k|^{q/p} \leq \bigl(\sum_{k\in \Z} |a_k|\bigr)^{q/p}$ for all real-valued sequences $(a_k)_{k\in\Z}$, we obtain
\begin{align*}
\int_B \lvert v(x)\rvert^q \,dx
&\le  \sum_{k\in \Z} 2^{(k+1)q} \mu(B\cap A_{k+1}) 
\le\sum_{k\in\Z}  2^{(k+1)q} \mu(\{x\in B : v_{2^{k-1}}^{2^{k}}(x)> 2^{k-2}\})\\&\le
2^{1+4q}C_1\Biggl(\sum_{k\in \Z} \int_{\lambda B}\int_{\lambda B} \vert v_{2^{k-1}}^{2^{k}}(y)-v_{2^{k-1}}^{2^{k}}(z)\vert^p K(y,z)\,dz\,dy\Biggr)^\frac{q}{p}.
\end{align*}
Using the definition of $v_{2^{k-1}}^{2^{k}}$, we can now estimate
\begin{equation}\label{e.remaining}
\begin{split}
\sum_{k\in \Z} & \int_{\lambda B}\int_{\lambda B} \vert v_{2^{k-1}}^{2^{k}}(y)-v_{2^{k-1}}^{2^{k}}(z)\vert^p K(y,z)\,dz\,dy\\
& \le \Biggl\{ \sum_{k\in \Z}  \sum_{ -\infty \le i\le k} \sum_{j\ge k} \int_{A_i}\int_{A_j}
+  \sum_{k\in \Z}   \sum_{i\ge k} \sum_{-\infty\le j\le k} \int_{A_i}\int_{A_j}
\Biggr\}
\vert v_{2^{k-1}}^{2^{k}}(y)-v_{2^{k-1}}^{2^{k}}(z)\vert^p K(y,z)\,dz\,dy.
\end{split}
\end{equation}
Let $y\in A_i$ and $z\in A_j$, where $j-1> i \ge -\infty$, and let $k\in\Z$.
Then 
\[
\lvert v(y)-v(z)\rvert \geq \lvert v(z)\rvert  - \lvert v(y)\rvert  \geq 2^{j-2}
\] 
and $\lvert v_{2^{k-1}}^{2^{k}}(y)-v_{2^{k-1}}^{2^{k}}(z)\rvert \le 2^k$,
and so \ed
\begin{equation}\label{summandEstimate}
\lvert v_{2^{k-1}}^{2^{k}}(y)-v_{2^{k-1}}^{2^{k}}(z)\rvert \le 4\cdot 2^{k-j}\lvert v(y)-v(z)\rvert.
\end{equation}
On the other hand,
the estimate
\[
\lvert v_{2^{k-1}}^{2^{k}}(y)-v_{2^{k-1}}^{2^{k}}(z)\rvert  \leq \lvert v(y)-v(z)\rvert
\]
holds for every $k\in\Z$, 
and thus we conclude that 
inequality \eqref{summandEstimate} holds whenever $-\infty\le i\le k\le j$
and $(y,z)\in A_i\times A_j$.

By inequality \eqref{summandEstimate}, we have
\begin{equation}\label{ThreeSums}
\begin{split}
\sum_{k\in \Z}\sum_{-\infty\le i\le k} &\sum_{j\ge k} \int_{A_i}\int_{A_j}\vert v_{2^{k-1}}^{2^{k}}(y)-v_{2^{k-1}}^{2^{k}}(z)\vert^p K(y,z)\,
dz\,dy\\
 &\le  4^p \sum_{k\in \Z} \sum_{-\infty\le i\leq k} \sum_{j\geq k} 2^{p(k-j)} \int_{A_i} \int_{A_j} \lvert v(y)-v(z)\rvert^pK(y,z)\,dz\,dy.
\end{split}
\end{equation}
Since $\sum_{k=i}^j 2^{p(k-j)}  \le (1-2^{-p})^{-1}$, changing the order of the summation shows
that the right-hand side of inequality \eqref{ThreeSums} is bounded by
\[
 \frac{4^p}{1-2^{-p}}
\int_{\lambda B} \int_{\lambda B} \lvert v(y)-v(z)\rvert^p K(y,z)\,dz\,dy.
\]
 The second sum on the right-hand side of \eqref{e.remaining} 
can be estimated in the same way. 
To conclude that (B) holds with $C_2=C(p,q)C_1$ it remains to recall that $\lvert u-b\rvert=v_{+}+v_{-}$ and $q>0$. Observe also that
$\lvert v_\pm(y)-v_\pm(z)\rvert \le \lvert u(y)-u(z)\rvert$ for all $y,z\in \lambda B$. 
\end{proof}

We also need certain maximal functions.
If $B\subset X$ is an (open) ball and $u\in L^1(B)$, then the noncentred maximal function restricted to
$B$ is
\[
M_B^* u(x)=\sup_{B'} \intav_{B'} \lvert u(y)\rvert\,dy,
\]
where the supremum is taken over all balls $B'\subset B$ containing $x\in B$.
We will apply the following lemma from \cite[Lemma~3.12]{MR2867756}.

\begin{lemma}\label{l.max}
Let $B\subset X$ be a ball and let $u\in L^1(B)$. Then $M_B^*u$ is lower
semicontinuous in $B$ and satisfies
\[
\mu(E_\tau)\le \frac{c_D^3}{\tau} \int_{E_\tau} \lvert u(x)\rvert\,dx\quad \text{ and } \quad
\lim_{\tau \to\infty} \tau \mu(E_\tau)=0,
\]
where $E_\tau=\{x\in B :  M^*_Bu(x)>\tau\}$ and $\tau>0$.
\end{lemma}

The next theorem  gives a sufficient condition for
the fractional $(s,q,p,p)$-Poincar\'e inequality
with $q=p^*=Qp/(Q-sp)$. 
The proof is essentially the same as the  
argument in \cite[pp.\ 95--97]{MR2867756},
but we present the details
for the sake of completeness.
In particular, the fractional Maz$'$ya truncation method is
needed with sufficiently careful tracking of the constants.
Recall that we assume throughout that $\mu$ is doubling, with constant $c_D\ge 1$ in~\eqref{e.doubling}.
Hence, for any fixed $1\le p<\infty$ and $0<s<1$ there exists
an exponent $Q>sp$ such that~\eqref{e.doubling_quant}
holds, and then $p^*=Qp/(Q-sp)>p$.
The exponent $Q$ in~\eqref{e.doubling_quant} is not uniquely determined, and
a smaller value of $Q>sp$ gives in Theorem~\ref{t.qp_impro} a larger exponent $p^*$, which 
in turn yields a stronger version of the Sobolev--Poincar\'e inequality.

\begin{theorem}\label{t.qp_impro}
Assume that $\mu$ is reverse doubling, with
constants $\sigma>0$ and $c_\sigma>0$ in~\eqref{e.reverse_doubling},
and 
let $Q>0$ and $c_Q>0$ be the constants in~\eqref{e.doubling_quant}.
Let $1\le p<\infty$ and $0<s<1$ be such that $sp<Q$,
and let $p^*=Qp/(Q-sp)$.
Then $X$ supports a $(s,p^*,p,p)$-Poincar\'e inequality,
with constants $\lambda=2$
and $c_P=c_P(Q,p,s,\sigma,c_D,c_Q,c_\sigma)$.
\end{theorem}

\begin{proof}
Let $B=B(x_0,r)$ be a ball in $X$ and let $u\in L^\infty(2B)$.
It suffices to prove that there exists a constant $C=C(Q,p,s,\sigma,c_D,c_Q,c_\sigma)$ such that
\begin{equation}\label{e.mazya_goal}
\begin{split}
\mu(\{x\in B :  \lvert u(x)-u_{2B}\rvert > t\})\, t^{p^*}
\le Cr^{sp^*}\mu(B)^{1-p^*/p}\biggl( \int_{2B} g_{u,s,p,2B}(y)^p
\,dy\biggr)^{\frac{p^*}{p}}
\end{split}
\end{equation}
whenever $t>0$. Then the 
$(s,p^*,p,p)$-Poincar\'e inequality follows from 
Theorem~\ref{t.truncation},  applied 
with the kernel 
\[
K(y,z)=\frac{1}{d(y,z)^{sp}\mu(B(y,d(y,z)))},\qquad y,z\in X,
\]
together with the doubling property of $\mu$
and the inequality 
\[
\int_B \lvert u(x)-u_B\rvert^p\,dx \leq 2^p \inf_{a\in\R} \int_B\vert u(x)-a\vert ^{p}\,dx,
\]
which in turn follows from H\"older's  inequality.

We prove \eqref{e.mazya_goal} for a fixed $t>0$. We may assume that 
$r< 2\diam(X)$ and 
\begin{equation}\label{e.integral}
0<\int_{2B} g_{u,s,p,2B}(y)^p\,dy<\infty.
\end{equation}
Indeed, if the integral in \eqref{e.integral} vanishes, then $u$ is a constant almost everywhere 
in the ball  $B$ by the $(s,p,p,p)$-Poincar\'e inequality
given in Lemma~\ref{l.qp}. 
Write $B_0=2B$, $r_0=2r$ and $M=M_{B_0}^*((g_{u,s,p,2B})^p)$.
By \cite[Lemma~1.8]{MR1800917}, $\mu$-almost every point $x\in B$ is a Lebesgue point of $u$.
Lemma \ref{l.max} implies that the function $M$ is finite
$\mu$-almost everywhere in $B$.

Let $x\in B$ be a Lebesgue point of $u$, with $M(x)<\infty$, and 
write $r_j=2^{-j}r$ and $B_j=B(x,r_j)$, for $j=1,2,\ldots$.
By the doubling property of $\mu$ and the $(s,1,p,p)$-Poincar\'e inequality in Lemma~\ref{l.qp},
\begin{align*}
\lvert u(x)-u_{B_0}\rvert&=\lim_{k\to \infty} \lvert u_{B_k}-u_{B_0}\rvert=\lim_{k\to\infty} \left\lvert \sum_{j=0}^{k-1} (u_{B_{j+1}}-u_{B_j})\right\rvert\\
&\le \sum_{j=0}^\infty \intav_{B_{j+1}} \lvert u(y)-u_{B_j}\rvert\, dy
\le c_D^3\sum_{j=0}^\infty \intav_{B_{j}} \lvert u(y)-u_{B_j}\rvert\, dy\\
&\le C(p,s,c_D)\sum_{j=0}^\infty  r_j^{s} \biggl(\intav_{B_j} 
g_{u,s,p,B_j}(y)^p
\,dy\biggr)^{\frac{1}{p}}
\\&\le C(p,s,c_D)\sum_{j=0}^\infty  r_j^{s} \biggl(
\intav_{B_j} g_{u,s,p,2B}(y)^p
\,dy
\biggr)^{\frac{1}{p}}.
\end{align*}
Condition \eqref{e.doubling_quant}, applied to the balls $B_j\subset B_0$  on the right-hand side, gives
\begin{equation}\label{eq.sigmat}
\lvert u(x)-u_{B_0}\rvert \le  C(Q,p,s,c_D,c_Q) \frac{r^s}{\mu(B_0)^{s/Q}}\underbrace{\sum_{j=0}^\infty  \mu(B_j)^{s/Q-1/p} 
\biggl(\int_{B_j} g_{u,s,p,2B}(y)^p
\,dy\biggr)^{\frac{1}{p}}}_{\Sigma'+\Sigma''}.
\end{equation}
We write the sum in~\eqref{eq.sigmat}  as $\Sigma'+\Sigma''$, where the summations are
over $0\le j<j_0$ and $j\ge j_0$, respectively, and the cut-off number $j_0\in \N$ 
is chosen as follows (depending on $x$).
Since $B_0\subset 8B_1$ and 
\[
0<\intav_{B_0}g_{u,s,p,2B} (y)^p\,dy\le M(x)<\infty,
\] 
there exists $j_0\ge 1$ such that
\begin{equation}\label{e.valid}
c_D^2\, \mu(B_{j_0}) \le \frac{1}{M(x)}  \int_{B_0} g_{u,s,p,2B}(y)^p\,dy\le c_D^3\, \mu(B_{j_0}).
\end{equation}
More precisely, by \eqref{e.reverse_doubling} $\mu(B_{j})\to 0$ as $j\to\infty$, and hence 
we can choose the largest integer $j_0$ for which the right inequality holds. The
left inequality then follows from the doubling property of~$\mu$. 


In the first sum $\Sigma'$ we 
have $\mu(B_j)\ge c_\sigma^{-1} 2^{\sigma(j_0-j)} \mu(B_{j_0})$ 
for every $0\le j<j_0$,
by~\eqref{e.reverse_doubling}.
Since $s/Q-1/p<0$, we obtain 
\begin{align*}
\Sigma'&=\sum_{j=0}^{j_0-1}  \mu(B_j)^{s/Q-1/p} \biggl(
\int_{B_j} g_{u,s,p,2B}(y)^p
\,dy
\biggr)^{\frac{1}{p}}\\
&\le C(Q,p,s,c_\sigma)\,  \mu(B_{j_0})^{s/Q-1/p} \biggl(
\int_{B_0} g_{u,s,p,2B}(y)^p
\,dy
\biggr)^{\frac{1}{p}} \sum_{j=0}^{j_0-1} 2^{\sigma(j_0-j)(s/Q-1/p)},\\ 
& \le C(Q,p,s,\sigma,c_D,c_\sigma)\, 
\biggl( \int_{B_0} g_{u,s,p,2B}(y)^p\,dy\biggr)^{\frac{s}{Q}} M(x)^{1/p-s/Q},
\end{align*}
where the sum on the second line is bounded from above by a constant 
$0<C(Q,p,s,\sigma)<\infty$ that can be chosen to be independent of $j_0$,
and the last step follows from the right-hand inequality in~\eqref{e.valid}.

Correspondingly,
in the second sum $\Sigma''$ 
we have $\mu(B_j)\le c_\sigma 2^{\sigma(j_0-j)} \mu(B_{j_0})$ for every $j\ge j_0$, by~\eqref{e.reverse_doubling}.
Using also the maximal function 
$M=M_{B_0}^* ((g_{u,s,p,2B})^p)$,
we obtain
\begin{align*}
\Sigma'' &=\sum_{j=j_0}^{\infty}  \mu(B_j)^{s/Q} \biggl(
\intav_{B_j} g_{u,s,p,2B}(y)^p \,dy\biggr)^{\frac{1}{p}}\\
& \le C(Q,s,c_\sigma) \mu(B_{j_0})^{s/Q}M(x)^{1/p} \sum_{j=j_0}^\infty 2^{\sigma(j_0-j)s/Q},\\
& \le C(Q,p,s,\sigma,c_D,c_\sigma)\, 
\biggl( \int_{B_0} g_{u,s,p,2B}(y)^p\,dy\biggr)^{\frac{s}{Q}} M(x)^{1/p-s/Q},
\end{align*}
where the last sum is bounded from above by a constant $0<C(Q,s,\sigma)<\infty$
and the final step follows from the left-hand inequality in~\eqref{e.valid}.

Substituting the above estimates for $\Sigma'$ and $\Sigma''$ to~\eqref{eq.sigmat} gives
\begin{align*}
\lvert u(x)-u_{B_0}\rvert &\le C(Q,p,s,c_D,c_Q) \frac{r^s}{\mu(B_0)^{s/Q}}(\Sigma'+\Sigma'')\\
&\le C(Q,p,s,\sigma,c_D,c_Q,c_\sigma)\, r^s \biggl( \intav_{B_0} g_{u,s,p,2B}(y)^p\,dy\biggr)^{\frac{s}{Q}} M(x)^{\frac{1}{p^*}},
\end{align*}
for $p^*=Qp/(Q-sp)$.
In particular, if  $\lvert u(x)-u_{B_0}\rvert > t>0$, then 
\begin{align*}
M(x)> C(Q,p,s,\sigma,c_D,c_Q, c_\sigma )\,\frac{t^{p^*}}{r^{sp^*}}\biggl( \intav_{B_0}g_{u,s,p,2B}(y)^p\,dy\biggr)^{-\frac{sp^*}{Q}}
=\tau(t)>0.
\end{align*}
From this estimate, which is valid for $\mu$-almost every $x\in B$, and Lemma \ref{l.max}, we obtain
\begin{align*}
\mu(\{x\in B :  \lvert u(x)-u_{B_0}\rvert > t\})\, t^{p^*}&\le \mu(\{x\in B_0 :  M(x) > \tau(t)\})\,t^{p^*}
\le \frac{c_D^3 t^{p^*}}{\tau(t)}\int_{B_0} g_{u,s,p,2B}(x)^p\,dx
\\& \le C(Q,p,s,\sigma,c_D,c_Q,c_\sigma)\,r^{sp^*}\mu(B)^{1-p^*/p}\biggl( \int_{B_0} g_{u,s,p,2B}(x)^p\,dx\biggr)^{\frac{p^*}{p}}
\end{align*}
for every $t>0$. Inequality \eqref{e.mazya_goal} follows, and the proof is complete.
\end{proof}

\section{Capacitary and boundary Poincar\'e inequalities}\label{s.capacitary}

Next we study versions of fractional Poincar\'e inequalities, in which the zero sets of
functions are taken into account. As a tool we will apply a variant  of the 
fractional relative capacity, 
compare to~\cite[Definition~7.1]{MR3605166} and see  also \cite{DydaVahakangas2015} and \cite[\S 11]{Mazya2011}.

\begin{definition}\label{d.capacity}
Let $0<s<1$, $1\le t,p<\infty$, and $\Lambda\ge 2$. 
Let $B\subset X$ be a ball and let $E\subset \iol{B}$ be a closed set. Then we write
\[
\Cp_{s,p,t} (E,2B,\Lambda B) = \inf_\varphi \int_{\Lambda B} \biggl( \int_{\Lambda B}\frac{|\varphi(x)-\varphi(y)|^t}{d(x,y)^{st}\mu(B(x,d(x,y)))}\,dy\biggr)^{p/t}\,dx
\]
where the infimum is taken over all continuous functions 
$\varphi\colon X\to \R$ that are integrable on balls, such that $\varphi(x)\ge 1$ for every $x\in E$ and 
$\varphi(x)=0$ for every $x\in X\setminus 2B$.
\end{definition}

 The following simple lemma is needed in the proof of Theorem~\ref{t.Mazya}.

\begin{lemma}\label{l.alpha}
Let $\alpha>0$.
There is a constant $C(\alpha,c_D)>0$ such that
\[
  r^{-\alpha}\int_{B(x,r)} \frac{d(x,y)^{\alpha}}{\mu(B(x,d(x,y)))} \, dy\leq C(\alpha,c_D)
\]
for every $x\in X$ and $r>0$.
\end{lemma}

\begin{proof}
Let $x\in X$ and $r>0$.
For each $j\in \{0,1,\ldots\}$ we write
\[A_j(x,r)=\{y\in X :  2^{-j-1}r\le d(x,y)<2^{-j}r\}.\]
By the doubling condition \eqref{e.doubling} of the measure $\mu$ and the standing assumption that $\mu(\{x\})=0$,
we obtain 
\begin{align*}
\int_{B(x,r)} \frac{d(x,y)^{\alpha}}{\mu(B(x,d(x,y)))} \, dy 
&= \sum_{j=0}^\infty \int_{A_j(x,r)} \frac{d(x,y)^{\alpha}}{\mu(B(x,d(x,y)))} \, dy \\
&\le \sum_{j=0}^\infty (2^{-j}r)^{\alpha}\frac{\mu(A_j(x,r))}{\mu(B(x,2^{-j-1}r))}\\
&\le \sum_{j=0}^\infty (2^{-j}r)^{\alpha}\frac{\mu(B(x,2^{-j}r))}{\mu(B(x,2^{-j-1}r))}\le C(\alpha,c_D)\, r^\alpha.\qedhere
\end{align*}
\end{proof}

The next result is a fractional version of Maz$'$ya's capacitary Poincar\'e inequality,
compare to~\cite[Theorem~6.21]{MR2867756}.
The argument is similar to that in~\cite{MR2867756}, but 
there are several technical differences due to  
the present non-local  setting.

\begin{theorem}\label{t.Mazya}
Let $q\geq p \geq 1$, $0<s<1$, $1\leq t < \infty$ and $\Lambda\ge 2$. 
Assume that $X$ supports a $(s,q,p,t)$-Poincar\'e inequality
with constants $c_P>0$ and $\lambda\ge 1$.
Let $u\colon X\to \R $ be a continuous function and let
\[
Z=\{x\in X : u(x)=0 \}.
\] 
Then, for all balls $B=B(x_0,r)\subset X$,
\begin{equation}\label{e.fo}\begin{split}
\biggl( \intav_{\Lambda B} |u(x)|^q\, dx  \biggr)^{p/q}
\le \frac{C(s,t,p,c_D,c_P,\Lambda)}{\Cp_{s,p,t}(\iol{B}\cap Z, 2B,\Lambda B)}
  \int_{\lambda\Lambda B} g_{u,s,t,\lambda\Lambda B}(x)^p \,dx.
  \end{split}
\end{equation}
\end{theorem}

\begin{proof}
By replacing $u$ with $u_k=\min\{\lvert u\rvert,k\}$,  for $k\in\N$, applying Fatou's lemma, 
and using inequalities $g_{u_k,s,t,\lambda\Lambda B}\le g_{u,s,t,\lambda\Lambda B}$, we may assume that $u\geq 0$
and that $u$ is bounded.
Fix a ball $B=B(x_0,r)$ in $X$.
Without loss of generality we may assume that 
the right-hand side of inequality \eqref{e.fo} is finite. Let
\[
  \overline{u} = \biggl( \intav_{\Lambda B} |u(x)|^q\, dx \biggr)^{1/q}<\infty.
\]
We may assume that $\overline{u}>0$, as otherwise there is nothing to prove. 
  
Let $\eta(x)=\max\{0,1-\dist(x,B)/r\}$ for every $x\in X$. Then 
\[|\eta(x)-\eta(y)|\leq d(x,y)/r,\qquad \text{ for every }x,y\in X, \] 
$0\leq \eta \leq 1$ in $X$, $\eta=1$ in $\iol{B}$ and $\eta=0$ outside $2B$. The function 
$\varphi=(1-u/\overline{u}) \eta$ is  bounded
and continuous,  
$\varphi=1$ in $\iol{B}\cap Z$,
and $\varphi=0$ outside $2B$.
By Definition~\ref{d.capacity} of the capacity,   we have 
\[\begin{split}
  \Cp_{s,p,t}(\iol{B}\cap Z, 2B,\Lambda B) &\leq \int_{\Lambda B} \biggl( \int_{\Lambda B} \frac{|\varphi(x)-\varphi(y)|^t}{d(x,y)^{st}\mu(B(x,d(x,y)))}\,dy\biggr)^{p/t}\,dx \\
  &=  \frac{1}{\overline{u}^p}\int_{\Lambda B} \biggl( \int_{\Lambda B} \frac{|\eta(x)(\overline{u}-u(x)) -\eta(y)(\overline{u}-u(y))|^t}{d(x,y)^{st}\mu(B(x,d(x,y)))} \,dy \biggr)^{p/t} \!\!\! dx\\
  &=  \frac{1}{\overline{u}^p} I.
\end{split}\]
To estimate $I$, we write
\[\begin{split}
    I &= \int_{\Lambda B} \biggl( \int_{\Lambda B}
    \frac{|\eta(x)(\overline{u}-u(x)) -\eta(y)(\overline{u}-u(x)) + \eta(y)(\overline{u}-u(x)) - \eta(y)(\overline{u}-u(y))|^t}{d(x,y)^{st}\mu(B(x,d(x,y)))}\,dy\biggr)^{p/t} dx \\
    &\leq C(t,p) \int_{\Lambda B} |\overline{u}-u(x)|^p \biggl( \int_{\Lambda B} \frac{|\eta(x)-\eta(y)|^t}{d(x,y)^{st}\mu(B(x,d(x,y)))}\,dy\biggr)^{p/t} dx \\&\qquad\qquad\qquad\qquad\qquad\qquad +
     C(t,p) \int_{\Lambda B}\biggl( \int_{\Lambda B} \eta(y)^t \frac{|u(y)-u(x)|^t}{d(x,y)^{st}\mu(B(x,d(x,y)))}\,dy \biggr)^{p/t} dx.
\end{split}\]
  Fix $x\in \Lambda B$. Since $|\eta(x)-\eta(y)| \leq  d(x,y)/r$
  for each $y\in \Lambda B$, by Lemma \ref{l.alpha} we have
  \[
    \int_{\Lambda B} \frac{|\eta(x)-\eta(y)|^t}{d(x,y)^{st}\mu(B(x,d(x,y)))}\,dy \leq
   r^{-t} \int_{B(x,2\Lambda r)} \frac{d(x,y)^{t(1-s)}}{\mu(B(x,d(x,y)))}\,dy \leq C(s,t,c_D,\Lambda) r^{-st}.
  \]
  Taking also into account that $0\le \eta^t\le 1$ in $\Lambda B$, we obtain
  \[
    I \leq C(s,t,p,c_D,\Lambda)r^{-sp} \int_{\Lambda B} |\overline{u}-u(x)|^p\,dx
    + C(t,p) \int_{\Lambda B}g_{u,s,t,\Lambda B}(x)^p\, dx.
  \]
Hence, we are left with estimating the following integral, with $a=(q-p)/(pq)$,
\[\begin{split}
    \biggl(\int_{\Lambda B} |\overline{u}-u(x)|^p\,dx \biggr)^{1/p} &\leq 
    \mu(\Lambda B)^{a}\biggl(\int_{\Lambda B} |\overline{u}-u(x)|^q\,dx \biggr)^{1/q}\\&\le
      \mu(\Lambda B)^{a}  \biggl(\int_{\Lambda B} |u(x)-u_{\Lambda B}|^q\,dx \biggr)^{1/q} + |\overline{u}-u_{\Lambda B}| \mu(\Lambda B)^{a+1/q}.
\end{split}\]
 The  first step above relies on the assumption $q\ge p$. The right-hand side  can be estimated exactly as 
in \cite[pp.~144--145]{MR2867756}.   
Indeed, the second term may be estimated by the first one, since 
\begin{align*}
\lvert \overline{u}-u_{\Lambda B}\rvert\mu(\Lambda B)^{a+1/q}
&=\mu(\Lambda B)^{a}\big\lvert  \lVert u\rVert_{L^q(\Lambda B)} - \lVert u_{\Lambda B}\rVert_{L^q(\Lambda B)}  \big\rvert\\
&\le \mu(\Lambda B)^{a}\lVert u-u_{\Lambda B}\rVert_{L^q(\Lambda B)}=\mu(\Lambda B)^{a}\biggl(\int_{\Lambda B} \lvert u(x)-u_{\Lambda B}\rvert^q\,dx\biggr)^{1/q}.
\end{align*}
The first term is in turn estimated by the assumed $(s,q,p,t)$-Poincar\'e inequality, 
\begin{align*}
\mu(\Lambda B)^{a}\biggl(\int_{\Lambda B} \lvert u(x)-u_{\Lambda B}\rvert^q\,dx\biggr)^{1/q}
&=\mu(\Lambda B)^{1/p}\biggl(\intav_{\Lambda B} \lvert u(x)-u_{\Lambda B}\rvert^q\,dx\biggr)^{1/q} \\
& \le c_P r^s \biggl( \int_{\lambda\Lambda B}g_{u,s,t,\lambda\Lambda B}(x)^p\,dx\biggr)^{1/p}.
\end{align*}   
This results in
\begin{align*}
I 
&\le C(s,t,p,c_D,c_P,\Lambda) \int_{\lambda\Lambda B} g_{u,s,t,\lambda\Lambda B}(x)^p\,dx,
\end{align*}
and it follows that
\begin{align*}
\biggl( \intav_{\Lambda B} |u(x)|^q dx \biggr)^{p/q}=\overline{u}^p
\le \frac{C(s,t,p,c_D,c_P,\Lambda)}{\Cp_{s,p,t}(\iol{B}\cap Z,2B,\Lambda B)}
\int_{\lambda\Lambda B}g_{u,s,t,\lambda\Lambda B}(x)^p\,dx,
\end{align*}
as required.
\end{proof}

Next we consider two notions that are closely related to the relative capacity
but have a more geometric flavor.  
The following concept of (co)dimension was introduced in~\cite{KaenmakiLehrbackVuorinen2013}. 

\begin{definition}\label{def.ucodima}
Let $E\subset X$. For $r>0$, the
open $r$-neighborhood of $E$
is the set 
\[E_r=\{x\in X : \dist(x,E)<r\}.\]
The upper Assouad codimension of $E$,
denoted by $\ucodima(E)$, is
the infimum of all $Q\ge 0$ for
which there is a constant $c>0$ such that
\[
\frac{\mu(E_r\cap B(x,R))}{\mu(B(x,R))}\ge c\Bigl(\frac{r}{R}\Bigr)^Q
\]
for every $x\in E$ and all $0<r<R<\diam(E)$. 
If $E$ consists of one point, then
the restriction $R<\diam(E)$ is removed.
\end{definition}

If the measure $\mu$ is $Q$-regular, then
$\ucodima(E)=Q-\ldima(E)$
for all $E\subset X$, where $\ldima(E)$ is the
lower (Assouad) dimension of $E$;
see~\cite[(3.11)]{KaenmakiLehrbackVuorinen2013}.
In the Euclidean space $\R^n$, which is regular with $Q=n$, the connection between fractional Hardy inequalities and
the lower Assouad dimension (as well as its dual, the upper Assouad dimension)
has been considered in~\cite{DydaKijaczko, DydaVahakangas2014};
see also~\cite{lehrbackHardyAssouad}.

We also need suitable versions of Hausdorff contents,
which give lower bounds for capacities, as in Lemma~\ref{l.codim} below.
In the case of non-fractional capacities, similar
ideas can be found for instance in~\cite[Theorem~5.9]{MR1654771}
and in several subsequent papers. 


\begin{definition}\label{def.Hcont}
The ($\rho$-restricted) Hausdorff content of codimension $\eta\ge 0$
is defined for sets $E\subset X$ by setting 
\[
\Ha^{\mu,\eta}_\rho(E)=\inf\Biggl\{\sum_{k} \mu(B(x_k,r_k))\,r_k^{-\eta} :
E\subset\bigcup_{k} B(x_k,r_k)\text{ and } 0<r_k\leq \rho \Biggr\}.
\]
\end{definition}

\begin{lemma}\label{l.codim}
Let $0<s<1$,  $1\le p,t<\infty$, $0\le \eta<p$ and $\Lambda> 2$.
Assume that $\mu$ is 
reverse doubling, with constants $\kappa=2/\Lambda$
and $0<c_R<1$ in~\eqref{e.rev_dbl_decay}. \ed
Let $B=B(x_0,r)\subset X$ be a ball with $r<\diam(X)/(2\Lambda)$,
and assume that $E\subset \iol{B}$ is a closed set.
Then
\[
\mathcal{H}^{\mu,s\eta}_{5\Lambda r}(E) \le 
C(s,t,p,\eta,c_R,c_D,\Lambda) r^{s(p-\eta)}\Cp_{s,p,t}(E,2B,\Lambda B).
\]
\end{lemma}

\begin{proof}
Fix $x\in E$ and write $B_0=\Lambda B=B(x_0,\Lambda r)$,
$r_0=\Lambda r$, $r_j=2^{-j+1}r$ and $B_j=B(x,r_j)$, $j=1,2,\ldots$.
Observe that there are test functions for $\Cp_{s,p,t}(E,2B,\Lambda B)$; let $\varphi$ be
any one of them.
By replacing $\varphi$ with $\max\{0,\min\{\varphi,1\}\}$, if necessary, we may assume that $0\le \varphi \le 1$. 
Thus $\varphi$ is continuous on $X$,  $\varphi=1$ on $E$, and $\varphi=0$ on $B_0\setminus 2B$. By inequality \eqref{e.rev_dbl_decay}, we have
\begin{align*}
0\le \varphi_{B_0}=\vint_{B_0} \varphi(y)\,dy \le \frac{\mu(2B)}{\mu(\Lambda B)}
\le c_R < 1.
\end{align*}
As a consequence, since $x\in E$, we find that
\[
\lvert \varphi(x)- \varphi_{B_0}\rvert  \ge 1-c_R>0.
\]
Let $\delta=s(p-\eta)/p>0$.
Proceeding as in the proof of Theorem \ref{t.qp_impro} with the $(s,1,p,t)$-Poincar\'e inequality given by Lemma \ref{l.qp}, we obtain 
\begin{align*}
\sum_{j=0}^\infty 2^{-j\delta}&=C(s,p,\eta,c_R)(1-c_R)\le C(s,p,\eta,c_R) \lvert \varphi(x)-\varphi_{B_0}\rvert
\\&\le C(s,t,p,\eta,c_R,c_D,\Lambda)\sum_{j=0}^\infty  r_j^{s} 
\biggl(\intav_{B_j} g_{\varphi,s,t,B_0}(y)^p\,dy\biggr)^{\frac{1}{p}}.
\end{align*}
In particular, there exists $j\in \{0,1,2,\ldots\}$,
depending on $x$, such that
\[
2^{-j\delta p}\le C(s,t,p,\eta,c_R,c_D,\Lambda) r_{j}^{sp} \intav_{B_{j}} g_{\varphi,s,t,B_0}(y)^p\,dy .
\]
Write $r_x=r_{j}$ and
$B_x=B(x,r_x)=B_{j}$. Then the previous estimate gives
\[
 \mu(B_x) r_x^{-s\eta}\le C(s,t,p,\eta,c_R,c_D,\Lambda)  r^{s(p-\eta)} \int_{B_x}  g_{\varphi,s,t,B_0}(y)^p\,dy.
\]

By the $5r$-covering lemma \cite[Lemma~1.7]{MR2867756}, we obtain points $x_k\in E$,
$k=1,2,\ldots$, such that the balls $B_{x_k}\subset B_0=\Lambda B$ 
with radii $r_{x_k}\le \Lambda r$ are
pairwise disjoint and 
$E\subset \bigcup_{k=1}^\infty 5B_{x_k}$. Hence,
\begin{align*}
\mathcal{H}^{\mu,s\eta}_{5\Lambda r}(E) &\le \sum_{k=1}^\infty  \mu(5B_{x_k})
(5r_{x_k})^{-s\eta}
 \le  C \sum_{k=1}^\infty  r^{s(p-\eta)} \int_{B_{x_k}}g_{\varphi,s,t,B_0}(x)^p\,dx\\
&\le  C r^{s(p-\eta)} \int_{\Lambda B}g_{\varphi,s,t,\Lambda B}(x)^p\,dx,
\end{align*}
where  $C=C(s,t,p,\eta,c_R,c_D,\Lambda)$. 
The desired inequality follows by taking infimum over all functions
$\varphi$ as above.
\end{proof}

The main result of this section 
is the following version of the fractional (Sobolev--)Poincar\'e inequality,
where the mean value $u_B$ can be omitted from the integral on the left-hand side. 
Due to the zero values on the set $E$, this kind of inequalities are often called
boundary Poincar\'e inequalities. 
The proof  below requires  completeness
of $X$ via \cite[Lemma~5.1]{lehrbackHardyAssouad}, which gives uniform lower
bounds for Hausdorff contents under the assumption that
$\ucodima(E)<sp$. Hence, in the forthcoming applications 
of Theorem~\ref{t.assouad} we also  make  
the assumption that the space $X$ is complete. 
Alternatively, in the following results the condition
$\ucodima(E)<sp$ could be replaced by an explicit
condition in terms of the relative capacity
or a suitable Hausdorff content, and then
the completeness assumption would not be needed.
However, in non-complete spaces one then has 
to add to Theorems~\ref{t.pointwise} and~\ref{t.integrated} also
the assumption that the continuous function $u$ is integrable on balls.

\begin{theorem}\label{t.assouad}
Let $q\geq p \geq 1$, $0<s<1$ and $1\leq t < \infty$. 
Assume that  the space $X$ is complete and 
supports a $(s,q,p,t)$-Poincar\'e inequality, with constants
$c_P$ and $\lambda\ge 1$,
and that $\mu$ is 
reverse doubling, with constants $0<\kappa<1$
and $0<c_R<1$ in~\eqref{e.rev_dbl_decay}. \ed
Let $E$ be a closed set with
$\ucodima(E)<sp$.
Then there is a constant $C>0$ 
such that
\begin{equation}\label{eq.bdry_poinc}
\biggl(\intav_{B} \lvert u(x)\rvert^q\,dx\biggr)^{p/q}
\le C R^{sp}\intav_{\lambda B} \biggl(\int_{\lambda B}\frac{|u(x)-u(y)|^t}{d(x,y)^{st}\mu(B(x,d(x,y)))}\,dy\biggr)^{p/t}\,dx
\end{equation}
whenever $u\colon X\to \R$ is a continuous function 
such that 
$u=0$ on $E$ and $B=B(w,R)$ is a ball with
$w\in E$ and $0<R<\diam(E)/2$.
\end{theorem} 

Note that \eqref{eq.bdry_poinc}  can be written as 
\[
\biggl(\intav_{B} \lvert u(x)\rvert^q\,dx\biggr)^{p/q}
\le 
C R^{sp}\intav_{\lambda B} 
g_{u,s,t,\lambda B}(x)^p\,dx.
\]

\begin{proof}
Fix a number $0\le \eta<p$ such that 
$\ucodima(E)<s \eta$, and let 
$w\in E$ and 
$0<R<\diam(E)/2$.
Write $\Lambda = 2/\kappa>2$ and $r=R/\Lambda<\diam(E)/(2\Lambda)\le \diam(X)/(2\Lambda)$.
We prove the claim \eqref{eq.bdry_poinc} for the ball $B(w,R)$, but for simplicity
we write during the proof that $B=B(w,r)=B(w,R/\Lambda)$. 

By a covering argument using the doubling condition and completeness of $X$, 
see \cite[Lemma~5.1]{lehrbackHardyAssouad}, we obtain
\begin{align*}
r^{-s\eta}\mu(B)&\le 
C\mathcal{H}^{\mu,s\eta}_{r}(\iol{B}\cap E)\le 
C\mathcal{H}^{\mu,s\eta}_{5\Lambda r}(\iol{B}\cap E)\\&
\le C r^{s(p-\eta)}\Cp_{s,p,t}(\iol{B}\cap E,2B,\Lambda B).
\end{align*}
Write $Z=\{y\in X :  u(y)=0\}\supset E$. By the monotonicity of capacity and the doubling condition we have
\[
\frac{1}{\Cp_{s,p,t}(\iol{B}\cap Z,2B,\Lambda B)}
\le \frac{1}{\Cp_{s,p,t}(\iol{B}\cap E,2B,\Lambda B)}
\le \frac{C r^{sp}}{\mu(B)}\le \frac{C R^{sp}}{\mu(\lambda\Lambda B)}.
\]
The desired inequality, for the ball $B(w,R)=B(w,\Lambda r)$, follows from Theorem~\ref{t.Mazya}.
\end{proof}

\begin{corollary}\label{cor.bdry_poinc}
Assume that $X$ is complete. 
Let $Q>0$ and $c_Q>0$ be the constants in~\eqref{e.doubling_quant},  and let
$q,p \geq 1$, $0<s<1$ and $1\leq t < \infty$
be such that either $q\le p\le t$, or $q\le p^*=Qp/(Q-sp)<\infty$ and $t=p$.
Assume that $\mu$ is 
reverse doubling, 
and let $E$ be a closed set with
$\ucodima(E)<sp$.
Then there is a constant $C>0$ such that
the boundary Poincar\'e inequality~\eqref{eq.bdry_poinc}
holds 
whenever $u\colon X\to \R$ is a continuous function such that
$u=0$ on $E$ and $B=B(w,R)$ is a ball with
$w\in E$ and $0<R<\diam(E)/2$.
\end{corollary}

\begin{proof}
By Lemma~\ref{l.qp}, $X$ supports a $(s,q,p,t)$-Poincar\'e inequality whenever $q\le\min\{p,t\}$.
In particular $X$ supports a $(s,p,p,t)$-Poincar\'e inequality whenever $p\le t$,
and for $q\le p\le t$ the claim then follows from Theorem~\ref{t.assouad} and H\"older's inequality on 
the left-hand side. 

On the other hand, $X$ supports a $(s,p^*,p,p)$-Poincar\'e inequality by Theorem~\ref{t.qp_impro}.
Theorem~\ref{t.assouad} gives the desired inequality for $q=p^*$ and $t=p$, and for $q\le p^*$ and $t=p$ 
the claim follows again from H\"older's inequality on the left-hand side.
\end{proof}

\section{Pointwise and integral Hardy inequalities}\label{s.hardy}

In this section we apply the  
Sobolev--Poincar\'e and boundary Poincar\'e inequalities from the previous sections 
to fractional Hardy-type inequalities involving distance weights. 
We begin with a pointwise version of the fractional  Hardy inequality, given in terms of 
the fractional maximal function. For $\alpha\in\R$ and a measurable function $u$ on $X$, 
this is defined as
\[
M_\alpha u(x)=\sup_{r>0}r^\alpha\vint_{B(x,r)} \lvert u(y)\rvert\,dy,\qquad \text{ for every }  x\in X.
\]
In particular, if $\alpha=0$, then $M_\alpha=M_0=M$ is the centered
Hardy--Littlewood maximal function.

\begin{theorem}\label{t.pointwise}
Let $\alpha\in\R$, $q\geq p \geq 1$, $\alpha<s<1$ and $1\leq t < \infty$.
Assume that $X$ is complete  and 
supports a $(s,q,p,t)$-Poincar\'e inequality
and that $\mu$ is 
reverse doubling. 
Let $E\subset X$ be a closed set with
$\ucodima(E)<sp$, and assume
that $u\colon X\to \R$ is a continuous function 
such that 
$u=0$ on $E$. 
Then there is a constant $C>0$, independent of $u$,  such that
\[
\lvert u(x)\rvert \le C d(x,E)^{s-\alpha} \bigl(M_{\alpha p}\bigl(\ch{B}(g_{u,s,t,B})^p\bigr)(x)\bigr)^{1/p}
\]
whenever $0<d(x,E)<\diam(E)$ 
and $B=B(x,2d(x,E))$.
\end{theorem}

\begin{proof}
Observe that the  continuous  function $u$ is integrable on balls since $X$ is complete,
see  \cite[Proposition 3.1]{MR2867756}.
Fix $x\in X$ with $0<d(x,E)<\diam(E)$ 
and let $B=B(x,2d(x,E))$.
Write $r=2d(x,E)>0$ and choose $w\in E$ such that $d(x,w)<(3/2)d(x,E)$. Then
\[
\widetilde B=B(w,r/(4\lambda))\subset B,
\]
where $\lambda\ge 1$ is the constant in the assumed $(s,q,p,t)$-Poincar\'e inequality, 
and 
\begin{equation}\label{eq.to_three}
\begin{split}
\lvert u(x)\rvert = \lvert u(x)-u_B+u_B - u_{\widetilde B}+u_{\widetilde B}\rvert
\le \lvert u(x)-u_B\rvert+\lvert u_B-u_{\widetilde B}\rvert + \lvert u_{\widetilde B}\rvert.
\end{split} 
\end{equation}
We estimate each of the terms on the right-hand side separately.

First observe \ed that $\lambda \widetilde B\subset B$
and that the measures of these two balls are comparable, 
with constants only depending on $c_D$. 
Hence, by applying Theorem~\ref{t.assouad} for the ball $\widetilde B$,
whose radius is $r/(4\lambda)<\diam(E)/2$, we obtain 
\begin{align*}
\lvert u_{\widetilde B}\rvert&\le
\biggl(\vint_{\widetilde B} \lvert u(y)\rvert^q\,dy\biggr)^{1/q} 
\le Cr^{s-\alpha}\biggl(r^{\alpha p}\intav_{\lambda \widetilde B} \ch{B}(y)
 g_{u,s,t,\lambda \widetilde B}(y)^p\,dy\biggr)^{1/p}\\
&\le Cr^{s-\alpha}\biggl(r^{\alpha p}\intav_{B} \ch{B}(y)g_{u,s,t,B}(y)^p\,dy\biggr)^{1/p}\\
&\le Cd(x,E)^{s-\alpha} \bigl(M_{\alpha p}\bigl(\ch{B}(g_{u,s,t,B})^p\bigr)(x)\bigr)^{1/p}.
\end{align*}
Recall from Lemma \ref{l.qp} that $X$ supports a $(s,1,p,t)$-Poincar\'e inequality, with constants $\lambda=1$ and $C(s,t,c_D)$.
By the doubling condition, followed by the $(s,1,p,t)$-Poincar\'e inequality, we obtain 
\begin{align*}
\lvert u_B-u_{\widetilde B}\rvert&\le C \vint_{B} \lvert u(y)-u_B\rvert\,dy
\le  Cr^{s-\alpha}\biggl(r^{\alpha p}\intav_{B} \ch{B}(y)g_{u,s,t,B}(y)^p\,dy\biggr)^{1/p}
\\&\le Cd(x,E)^{s-\alpha} \bigl(M_{\alpha p}\bigl(\ch{B}(g_{u,s,t,B})^p\bigr)(x)\bigr)^{1/p}.
\end{align*}

 In order to estimate the term $\lvert u(x)-u_B\rvert$, we write
$B_j=2^{-j}B=B(x,2^{-j}r)$ for $j=0,1,2,\ldots$.
Since   $\lim_{j\to \infty} u_{B_j}= u(x)$, we find that
\begin{align*}
\lvert u(x)-u_B\rvert&\le \sum_{j=0}^\infty \lvert u_{B_j}-u_{B_{j+1}}\rvert
 \le C\sum_{j=0}^\infty  \vint_{B_j} \lvert u(y)-u_{B_j}\rvert\,dy\\
&\le C \sum_{j=0}^\infty (2^{-j}r)^{s}\biggl(\intav_{{B_j}}
g_{u,s,t,B_j}(y)^p\,dy\biggr)^{1/p}\\
&\le C r^{s-\alpha}\sum_{j=0}^\infty 2^{-j(s-\alpha)}\biggl((2^{-j}r)^{\alpha p}\intav_{{B_j}}
\ch{B}(y)g_{u,s,t,B}(y)^p\,dy\biggr)^{1/p}\\
&\le C r^{s-\alpha}\sum_{j=0}^\infty 2^{-j(s-\alpha)} \bigl(M_{\alpha p}\bigl(\ch{B}(g_{u,s,t,B})^p\bigr)(x)\bigr)^{1/p}\\
&=C d(x,E)^{s-\alpha}\bigl(M_{\alpha p}\bigl(\ch{B}(g_{u,s,t,B})^p\bigr)(x)\bigr)^{1/p}.
\end{align*}
The claim follows from~\eqref{eq.to_three} and the estimates above. 
\end{proof}

Pointwise Hardy inequalities imply localized Hardy inequalities
for balls centered at $E$. Here we restrict ourselves to the case $q=p$.

\begin{theorem}\label{t.integrated}
Let $0<s<1$ and
$1<t<\infty$.
 Assume that $X$ is complete
and that $\mu$ is 
reverse doubling. 
Let $E\subset X$ be a closed set with
$\ucodima(E)<st$, and 
let $u\colon X\to\R$ be a continuous function such that
$u=0$ on $E$.
Then there is a constant $C>0$, independent of $u$, such that
\[
\int_{B\setminus E} \frac{\lvert u(x)\rvert^t}{d(x,E)^{st}}\,dx
\le C \int_{3B}\int_{3B}
\frac{|u(x)-u(y)|^t}{d(x,y)^{st}\mu(B(x,d(x,y)))}\,dy\,dx
\]
whenever $B=B(w,r)$ with $w\in E$ and $0<r<\diam(E)$.
\end{theorem}

\begin{proof}
Fix an exponent $1\le p<t$ in such a way that $\ucodima(E)<sp$. 
By Lemma \ref{l.qp} we find that $X$ supports
the $(s,p,p,t)$-Poincar\'e inequality with constants
$\lambda=1$ and $c_P=c_P(s,p,t,c_D)$.
Fix $x\in B\setminus E$. 
Then 
 $0<d(x,E)<r<\diam(E)$ and 
$B(x,2d(x,E))\subset 3B$.
Hence, by Theorem~\ref{t.pointwise} with $\alpha=0$ and $q=p$,
\begin{align*}
\frac{\lvert u(x)\rvert^t}{d(x,E)^{st}} &\le  C
 \bigl(M\bigl(\ch{B(x,2d(x,E))}(g_{u,s,t,B(x,2d(x,E))})^p\bigr)(x)\bigr)^{t/p}\\
&\le C
 \bigl(M\bigl(\ch{3B}(g_{u,s,t,3B})^p\bigr)(x)\bigr)^{t/p}.
\end{align*}
Integrating this inequality over the set $B\setminus E$ we obtain
\[
\int_{B\setminus E} \frac{\lvert u(x)\rvert^t}{d(x,E)^{st}}\,dx
\le C\int_X \bigl(M\bigl(\ch{3B}(g_{u,s,t,3B})^p\bigr)(x)\bigr)^{t/p}\,dx.
\]
Since $t>p$, the Hardy--Littlewood maximal theorem~\cite[Theorem~3.13]{MR2867756}
implies that
\[
\int_{B\setminus E} \frac{\lvert u(x)\rvert^t}{d(x,E)^{st}}\,dx
\le C\int_{3B} g_{u,s,t,3B}(x)^t\,dx.
\]
This concludes the proof.
\end{proof}

Next, we obtain a (partial) converse of Theorem~\ref{t.integrated}.
This shows that the dimensional condition $\ucodima(E) < st$ in Theorems~\ref{t.pointwise} 
and~\ref{t.integrated} is essentially sharp,
 up to the end point. The idea behind the proof
goes back to~\cite[Section~2]{Dyda2004}, where the impossibility of a
fractional Hardy inequality was shown in certain open sets 
$\Omega$ of the Euclidean space,
for instance if $\Omega$ is a Lipschitz domain and $st\le 1$
(note that in this case $\ucodima(\partial\Omega)=1$). 
On the other hand, a necessary condition for non-fractional
pointwise Hardy inequalities in metric spaces has been given 
in~\cite[Lemma~3]{KorteEtAl2011} in terms of a Hausdorff content
density condition.

\begin{theorem}\label{t.converse}
Let $0<s<1$, $1<t<\infty$, and $\lambda\ge 1$. Assume 
that $E\subset X$ is a (nonempty) closed set such that
\[
\int_{B\setminus E} \frac{\lvert u(x)\rvert^t}{d(x,E)^{st}}\,dx
\le C \int_{\lambda B}\int_{\lambda B}
\frac{\lvert u(x)-u(y)\rvert^t}{d(x,y)^{st}\mu(B(x,d(x,y)))}\,dy\,dx
\]
whenever $u\colon X\to\R$ is a bounded continuous function such that
$u=0$ on $E$,
and $B=B(w,r)$ with $w\in E$ and $0<r<\diam(E)$.
Then $\ucodima(E)\le st$.
\end{theorem}

\begin{proof}
Let $w\in E$ and $0<r<R_0<\diam(E)$. 
It suffices to show that
\begin{equation}\label{eq.codim_est}
\frac{\mu(E_r\cap B(w,R_0))}{\mu(B(w,R_0))}\ge c\Bigl(\frac{r}{R_0}\Bigr)^{st},
\end{equation}
where the constant $c$ is independent of $w$, $r$ and $R_0$.
For convenience, write $R=R_0/\lambda$ and $B=B(w,R)$, so that
$\lambda B=B(w,R_0)$.

If $\mu(E_r\cap B(w,R))\ge \frac 1 2 \mu(B(w,R))$, the claim is clear since then,
by doubling,
\[
\mu(E_r\cap B(w,R_0))\ge\mu(E_r\cap B(w,R))\ge\tfrac 1 2 \mu(B(w,R))\ge c \mu(B(w,R_0)),
\]
and on the other hand $\bigl(\frac{r}{R_0}\bigr)^{st}\le 1$. Thus we may assume that
$\mu(E_r\cap B(w,R)) < \tfrac 1 2 \mu(B(w,R))$, whence
\begin{equation}\label{eq.compl_meas}
\mu(B(w,R)\setminus E_r) \ge \tfrac 1 2 \mu(B(w,R))>0.
\end{equation}
Notice that then in particular $r < R=R_0/\lambda$
since otherwise $B(w,R)\setminus E_r=\emptyset$.

Let us now consider the continuous and bounded function $u\colon X\to\R$,
\[
u(x)=\min\{1,4r^{-1}d(x,E)\},\qquad x\in X.
\]
Then $u=0$ on $E$, $u=1$ in $X\setminus E_{r/4}$, and 
\[
\lvert u(x)-u(y)\rvert \le \min\bigl\{1, 4r^{-1}d(x,y)\bigr\} \quad \text{ for all } x,y\in X.
\]
Since $d(x,E)^{-st}\ge R^{-st}$ for  $x\in B(w,R)\setminus E_r$, we obtain
\begin{equation}\label{eq.lhs_lower}
\begin{split}
\int_{B\setminus E} \frac{\lvert u(x)\rvert^t}{d(x,E)^{st}}\,dx
& \ge \int_{B\setminus E_r} d(x,E)^{-st}\,dx
  \ge R^{-st}\mu(B(w,R)\setminus E_r)\\ 
& \ge \tfrac 1 2 R^{-st} \mu(B(w,R))\ge c R_0^{-st} \mu(B(w,R_0)),
\end{split}
\end{equation}
where the penultimate step follows from~\eqref{eq.compl_meas}
and the final inequality holds by doubling.

To prove the claim~\eqref{eq.codim_est}, it hence suffices to show that
\begin{equation}\label{eq.rhs_upper}
\int_{\lambda B}\int_{\lambda B}
\frac{\lvert u(x)-u(y)\rvert^t}{d(x,y)^{st}\mu(B(x,d(x,y)))}\,dy\,dx
\le C r^{-st}\mu(E_r\cap B(w,R_0)).
\end{equation}
Then~\eqref{eq.codim_est} follows directly from 
estimates~\eqref{eq.lhs_lower} and~\eqref{eq.rhs_upper} and the assumed
local fractional Hardy inequality.

Write 
\[
K(x,y)=\frac{\lvert u(x)-u(y)\rvert^t}{d(x,y)^{st}\mu(B(x,d(x,y)))}
\]
whenever $x,y\in \lambda B$, $x\not=y$.
Since $u(x)=1$ for $x\in \lambda B\setminus E_{r/4}$, and $K(x,y)\le c_DK(y,x)$ by doubling
for $x,y\in \lambda B$, $x\not=y$,
we have
\[ \begin{split}
\int_{\lambda B}\int_{\lambda B} K(x,y)\,dy\,dx 
& \le\int_{E_r\cap \lambda B}\int_{E_r\cap \lambda B} K(x,y) \,dy\,dx
    + (1+c_D)\int_{\lambda B\setminus E_r}\int_{E_{r/4}\cap\lambda B} K(x,y)\,dy\,dx\\
& =: I_1 + (1+c_D) I_2.
\end{split}
\]

Define
\[F_k = \{ x\in \lambda B : 2^{-k}\le d(x,E)<2^{-k+1}\}\]
and
\[A_j(x) = \{ y \in \lambda B : 2^{-j-1}\le d(y,x)<2^{-j}\},\]
for $k,j\in\Z$ and $x\in \lambda B$. Let also $k_1,k_2\in\Z$ be such that
\[
2^{-k_1-1}\le \lambda R<2^{-k_1} \quad \text{ and } \quad 2^{-k_2}\le r <2^{-k_2+1}.
\]

When $k\le k_2$ and $x\in F_k$, it holds that $E_{r/4}\cap A_j(x)=\emptyset$ for all $j \ge k+1$.
Using the estimate $\lvert u(x)-u(y)\rvert\le 1$ and changing also the orders
of summation and integration, we thus obtain
\[ \begin{split}
I_2 & \le \sum_{k=k_1}^{k_2}\int_{F_k} \sum_{j= k_1 - 1}^{k}\int_{E_{r/4}\cap A_j(x)}
\frac{1}{2^{-(j+1)st}\mu(B(x,d(x,y)))}\,dy\,dx\\
    & \le C \sum_{j=k_1-1}^{k_2} 2^{jst}  
    \int_{E_r\cap\lambda B} \sum_{k=j}^{k_2} \int_{\{x\in F_k : y\in A_j(x)\}}
\frac{1}{\mu(B(x,d(x,y)))}\,dx\,dy.
\end{split}
\]
But if $y\in A_j(x)$, then $d(x,y) < 2^{-j}$ and so
\[
\bigcup_{k=j}^{k_2}\{x\in F_k : y\in A_j(x)\}\subset B(y,2^{-j}). 
\]
Since $B(y,d(x,y))\subset B(x,2d(x,y))$, we obtain by doubling that (still for $y\in A_j(x)$)
\[\begin{split}
\mu(B(y,2^{-j})) & \le c_D \mu(B(y,2^{-j-1})) \le c_D \mu(B(y,d(x,y)))\\
& \le c_D \mu(B(x,2d(x,y)))\le c_D^2 \mu(B(x,d(x,y))).
\end{split}\]
We conclude 
that
\[ \begin{split}
I_2 
& \le C \sum_{j=k_1-1}^{k_2} 2^{jst}  
    \int_{E_r\cap\lambda B} \int_{B(y,2^{-j})}
\frac{1}{\mu(B(y,2^{-j}))}\,dx\,dy\\
& \le C \sum_{j=k_1-1}^{k_2} 2^{jst} \mu(E_r\cap\lambda B)
  \le C 2^{k_2 st} \mu(E_r\cap\lambda B) \le C r^{-st} \mu(E_r\cap\lambda B).  
\end{split}
\]

On the other hand, since $\lvert u(x)-u(y)\rvert^t\le \min\{1,4^t r^{-t}d(x,y)^t\}$ for all $x,y\in\lambda B$,
integral $I_1$ can be estimated as follows: 
\begin{equation}\label{eq.I_1}
\begin{split}
I_1 & \le \int_{E_r\cap\lambda B} \sum_{j=k_1-1}^{\infty}\int_{E_r\cap A_j(x)}
\frac{\lvert u(x)-u(y)\rvert^t}{d(x,y)^{st}\mu(B(x,d(x,y)))}\,dy\,dx\\
    & \le \int_{E_r\cap\lambda B} \sum_{j=k_1-1}^{k_2}\int_{E_r\cap A_j(x)}
   \frac{1}{2^{-(j+1)st}\mu(B(x,d(x,y)))}\,dy\,dx\\
    & \qquad + C \int_{E_r\cap\lambda B} \sum_{j=k_2}^{\infty}\int_{E_r\cap A_j(x)}
    \frac{r^{-t}d(x,y)^t}{d(x,y)^{st}\mu(B(x,d(x,y)))}\,dy\,dx.
\end{split}
\end{equation}
In the first integral on the right-hand side of~\eqref{eq.I_1}
\[
\int_{E_r\cap A_j(x)}  \frac{1}{2^{-(j+1)st}\mu(B(x,d(x,y)))}\,dy
\le \frac{1}{2^{-(j+1)st}} \int_{A_j(x)} \frac{1}{\mu(B(x,2^{-j-1}))}\,dy \le C2^{jst}
\]
since the measures of the balls $B(x,2^{-j})$ and $B(x,2^{-j-1})$ are comparable,
while in the second integral on the right-hand side of~\eqref{eq.I_1}
\[\begin{split}
\sum_{j=k_2}^{\infty}\int_{E_r\cap A_j(x)} &
    \frac{r^{-t}d(x,y)^t}{d(x,y)^{st}\mu(B(x,d(x,y)))}\,dy
      \le \sum_{j=k_2}^{\infty} r^{-t} \int_{A_j(x)}
        \frac{d(x,y)^{t(1-s)}}{\mu(B(x,d(x,y)))}\,dy\\
    & \le r^{-t} \sum_{j=k_2}^{\infty} 2^{-jt(1-s)} \int_{B(x,2^{-j})}
        \frac{1}{\mu(B(x,2^{-j-1}))}\,dy\\
    & \le C r^{-t} 2^{-k_2 t(1-s)} \le  C r^{-t} r^{t(1-s)} = Cr^{-st}.
\end{split}\]
Here we had again a converging geometric series since $t(1-s)>0$.

Substituting the above two estimates to~\eqref{eq.I_1}, we obtain
\[
\begin{split}
I_1 & \le C \int_{E_r\cap\lambda B} \sum_{j=k_1-1}^{k_2} 2^{jst}\,dx + 
    C \int_{E_r\cap\lambda B} r^{-st}\,dx \\
    & \le C \int_{E_r\cap\lambda B}  \bigl(2^{k_2st} + r^{-st}\bigr)\,dx
      \le C r^{-st}\mu(E_r\cap\lambda B).
\end{split}
\]
As $\lambda B=B(w,R_0)$, we conclude that $I_1 + (1+c_D) I_2 \le C r^{-st}\mu(E_r\cap B(w,R_0))$,
and this proves the claim.
\end{proof}

\def\cprime{$'$}


\begin{thebibliography}{10}

\bibitem{MR2867756}
A.~Bj{\"o}rn and J.~Bj{\"o}rn.
\newblock {\em Nonlinear potential theory on metric spaces}, volume~17 of {\em
  EMS Tracts in Mathematics}.
\newblock European Mathematical Society (EMS), Z\"urich, 2011.

\bibitem{MR4018755}
M.~E. Cejas, I.~Drelichman, and J.~C. Mart\'{\i}nez-Perales.
\newblock Improved fractional {P}oincar\'{e} type inequalities on {J}ohn
  domains.
\newblock {\em Ark. Mat.}, 57(2):285--315, 2019.

\bibitem{MR4046971}
S.-K. Chua.
\newblock Embedding and compact embedding for weighted and abstract {S}obolev
  spaces.
\newblock {\em Pacific J. Math.}, 303(2):519--568, 2019.

\bibitem{Dyda2004}
B.~Dyda.
\newblock A fractional order {H}ardy inequality.
\newblock {\em Illinois J. Math.}, 48(2):575--588, 2004.

\bibitem{DydaEtAl2019}
B.~Dyda, L.~Ihnatsyeva, J.~Lehrb\"{a}ck, H.~Tuominen, and A.~V.
  V\"{a}h\"{a}kangas.
\newblock Muckenhoupt {$A_p$}-properties of distance functions and applications
  to {H}ardy-{S}obolev--type inequalities.
\newblock {\em Potential Anal.}, 50(1):83--105, 2019.

\bibitem{DIV}
B.~Dyda, L.~Ihnatsyeva, and A.~V. V\"{a}h\"{a}kangas.
\newblock On improved fractional {S}obolev-{P}oincar\'{e} inequalities.
\newblock {\em Ark. Mat.}, 54(2):437--454, 2016.

\bibitem{DydaKijaczko}
B.~Dyda and M.~Kijaczko.
\newblock On density of compactly supported smooth functions in fractional
  {S}obolev spaces.
\newblock arXiv:2104.08953, 2021.

\bibitem{DydaVahakangas2014}
B.~Dyda and A.~V. V\"{a}h\"{a}kangas.
\newblock A framework for fractional {H}ardy inequalities.
\newblock {\em Ann. Acad. Sci. Fenn. Math.}, 39(2):675--689, 2014.

\bibitem{DydaVahakangas2015}
B.~Dyda and A.~V. V\"{a}h\"{a}kangas.
\newblock Characterizations for fractional {H}ardy inequality.
\newblock {\em Adv. Calc. Var.}, 8(2):173--182, 2015.

\bibitem{EdmundsHurriSyrjanenVahakangas2014}
D.~E. Edmunds, R.~Hurri-Syrj\"{a}nen, and A.~V. V\"{a}h\"{a}kangas.
\newblock Fractional {H}ardy-type inequalities in domains with uniformly fat
  complement.
\newblock {\em Proc. Amer. Math. Soc.}, 142(3):897--907, 2014.

\bibitem{Hajlasz1999}
P.~Haj{\l}asz.
\newblock Pointwise {H}ardy inequalities.
\newblock {\em Proc. Amer. Math. Soc.}, 127(2):417--423, 1999.

\bibitem{MR1886617}
P.~Haj{\l}asz.
\newblock Sobolev inequalities, truncation method, and {J}ohn domains.
\newblock In {\em Papers on analysis}, volume~83 of {\em Rep. Univ.
  Jyv\"askyl\"a Dep. Math. Stat.}, pages 109--126. Univ. Jyv\"askyl\"a,
  Jyv\"askyl\"a, 2001.

\bibitem{MR1800917}
J.~Heinonen.
\newblock {\em Lectures on analysis on metric spaces}.
\newblock Universitext. Springer-Verlag, New York, 2001.

\bibitem{MR1654771}
J.~Heinonen and P.~Koskela.
\newblock Quasiconformal maps in metric spaces with controlled geometry.
\newblock {\em Acta Math.}, 181(1):1--61, 1998.

\bibitem{HeinonenKoskelaShanmugalingamTyson2015}
J.~Heinonen, P.~Koskela, N.~Shanmugalingam, and J.~T. Tyson.
\newblock {\em Sobolev spaces on metric measure spaces. An approach based on
  upper gradients}, volume~27 of {\em New Mathematical Monographs}.
\newblock Cambridge University Press, Cambridge, 2015.

\bibitem{HurriSyrjanenVahakangas2013}
R.~Hurri-Syrj\"{a}nen and A.~V. V\"{a}h\"{a}kangas.
\newblock On fractional {P}oincar\'{e} inequalities.
\newblock {\em J. Anal. Math.}, 120:85--104, 2013.

\bibitem{HurriSyrjanenVahakangas2015}
R.~Hurri-Syrj\"{a}nen and A.~V. V\"{a}h\"{a}kangas.
\newblock Fractional {S}obolev-{P}oincar\'{e} and fractional {H}ardy
  inequalities in unbounded {J}ohn domains.
\newblock {\em Mathematika}, 61(2):385--401, 2015.

\bibitem{Ihnatsyeva_et_al2014}
L.~Ihnatsyeva, J.~Lehrb\"{a}ck, H.~Tuominen, and A.~V. V\"{a}h\"{a}kangas.
\newblock Fractional {H}ardy inequalities and visibility of the boundary.
\newblock {\em Studia Math.}, 224(1):47--80, 2014.

\bibitem{KaenmakiLehrbackVuorinen2013}
A.~K\"aenm\"aki, J.~Lehrb\"ack, and M.~Vuorinen.
\newblock Dimensions, {W}hitney covers, and tubular neighborhoods.
\newblock {\em Indiana Univ. Math. J.}, 62(6):1861--1889, 2013.

\bibitem{KinnunenMartio1996}
J.~Kinnunen and O.~Martio.
\newblock The {S}obolev capacity on metric spaces.
\newblock {\em Ann. Acad. Sci. Fenn. Math.}, 21(2):367--382, 1996.

\bibitem{KorteEtAl2011}
R.~Korte, J.~Lehrb{\"a}ck, and H.~Tuominen.
\newblock The equivalence between pointwise {H}ardy inequalities and uniform
  fatness.
\newblock {\em Math. Ann.}, 351(3):711--731, 2011.

\bibitem{Lehrback2014}
J.~Lehrb\"{a}ck.
\newblock Weighted {H}ardy inequalities beyond {L}ipschitz domains.
\newblock {\em Proc. Amer. Math. Soc.}, 142(5):1705--1715, 2014.

\bibitem{lehrbackHardyAssouad}
J.~Lehrb\"ack.
\newblock Hardy inequalities and {A}ssouad dimensions.
\newblock {\em J. Anal. Math.}, 131:367--398, 2017.

\bibitem{LossSloane2010}
M.~Loss and C.~Sloane.
\newblock Hardy inequalities for fractional integrals on general domains.
\newblock {\em J. Funct. Anal.}, 259(6):1369--1379, 2010.

\bibitem{MR3605166}
H.~Luiro and A.~V. V\"{a}h\"{a}kangas.
\newblock Beyond local maximal operators.
\newblock {\em Potential Anal.}, 46(2):201--226, 2017.

\bibitem{Mazya2011}
V.~G. Maz{\cprime}ya.
\newblock {\em Sobolev spaces with applications to elliptic partial
  differential equations}, volume 342 of {\em Grundlehren der Mathematischen
  Wissenschaften [Fundamental Principles of Mathematical Sciences]}.
\newblock Springer, Heidelberg, augmented edition, 2011.

\end{thebibliography}
\end{document}